\documentclass[11pt]{article}

\usepackage{amsmath,amsthm,amssymb,xypic}

\theoremstyle{plain}
    \newtheorem{theorem}{Theorem}[section]
    \newtheorem{lemma}[theorem]{Lemma}
    \newtheorem{corollary}[theorem]{Corollary}
    \newtheorem{proposition}[theorem]{Proposition}

 \theoremstyle{definition}
    \newtheorem{definition}[theorem]{Definition}

    \newtheorem{remark}[theorem]{Remark}

\theoremstyle{remark}

\numberwithin{equation}{section}

 \DeclareMathOperator{\tr}{tr}

\DeclareMathOperator{\Ad}{Ad}

\DeclareMathOperator{\ind}{index}

\DeclareMathOperator{\End}{End}

\DeclareMathOperator{\ch}{ch}
\DeclareMathOperator{\Todd}{Todd}

\DeclareMathOperator{\reg}{reg}

\DeclareMathOperator{\rank}{rank}

\DeclareMathOperator{\Spin}{Spin}

\DeclareMathOperator{\SO}{SO}
\DeclareMathOperator{\SL}{SL}

 \DeclareMathOperator{\Ind}{Ind}
 \DeclareMathOperator{\ev}{ev}

         \DeclareMathOperator{\supp}{supp}

  \DeclareMathOperator{\DInd}{D-Ind}

\DeclareMathOperator{\ds}{ds}

\begin{document}


\newcommand{\myemph}{\emph}

\newcommand{\Spinc}{\Spin^c}

    \newcommand{\R}{\mathbb{R}}
    \newcommand{\C}{\mathbb{C}} 
    \newcommand{\N}{\mathbb{N}}
    \newcommand{\Z}{\mathbb{Z}} 
    \newcommand{\Q}{\mathbb{Q}}
    \newcommand{\bT}{\mathbb{T}}
    \newcommand{\bP}{\mathbb{P}}

\newcommand{\g}{\mathfrak{g}}
\newcommand{\h}{\mathfrak{h}}
\newcommand{\p}{\mathfrak{p}}
\newcommand{\kg}{\mathfrak{g}} 
\newcommand{\kt}{\mathfrak{t}}
\newcommand{\ka}{\mathfrak{a}}
\newcommand{\XX}{\mathfrak{X}}
\newcommand{\kh}{\mathfrak{h}} 
\newcommand{\kp}{\mathfrak{p}}
\newcommand{\kk}{\mathfrak{k}}
\newcommand{\km}{\mathfrak{m}}
\newcommand{\ksl}{\mathfrak{sl}}

\newcommand{\cA}{\mathcal{A}}
\newcommand{\cE}{\mathcal{E}}
\newcommand{\calL}{\mathcal{L}}
\newcommand{\calH}{\mathcal{H}}
\newcommand{\cO}{\mathcal{O}}
\newcommand{\cB}{\mathcal{B}}
\newcommand{\cK}{\mathcal{K}}
\newcommand{\cP}{\mathcal{P}}
\newcommand{\cN}{\mathcal{N}}
\newcommand{\calD}{\mathcal{D}}
\newcommand{\cC}{\mathcal{C}}
\newcommand{\calS}{\mathcal{S}}

\newcommand{\cCM}{\cC}
\newcommand{\PM}{P}
\newcommand{\DM}{D}
\newcommand{\LM}{L}
\newcommand{\vM}{v}

\newcommand{\ddt}{\left. \frac{d}{dt}\right|_{t=0}}

\newcommand{\Bigwedge}{\textstyle{\bigwedge}}

\newcommand{\ii}{\sqrt{-1}}

\newcommand{\Ubar}{\overline{U}}

\newcommand{\tilK}{\widetilde{K}}
\newcommand{\tilT}{\widetilde{T}}

\newcommand{\beq}[1]{\begin{equation} \label{#1}}
\newcommand{\eeq}{\end{equation}}

\newcommand{\Todo}{\textbf{To do}}

\newcommand{\mattwo}[4]{
\left( \begin{array}{cc}
#1 & #2 \\ #3 & #4
\end{array}
\right)
}

\newenvironment{proofof}[1]
{\noindent \emph{Proof of #1.}}{\hfill $\square$}

\title{Orbital integrals and $K$-theory classes}

\author{Peter Hochs\footnote{University of Adelaide, \texttt{peter.hochs@adelaide.edu.au}} {} and Hang Wang\footnote{East China Normal University, \texttt{wanghang@math.ecnu.edu.cn}; University of Adelaide, \texttt{hang.wang01@adelaide.edu.au}}}

\date{\today}

\maketitle

\begin{abstract}
%
Let $G$ be a semisimple Lie group with discrete series. We use maps  $K_0(C^*_rG)\to \C$ defined by orbital integrals to recover group theoretic information about $G$, including information contained in $K$-theory classes not associated to the discrete series. An important tool is a fixed point formula for equivariant indices obtained by the authors in an earlier paper.
Applications include a tool to distinguish classes in $K_0(C^*_rG)$, the (known) injectivity of Dirac induction, versions of Selberg's principle in $K$-theory and for matrix coefficients of the discrete series, a Tannaka-type duality, and a way to extract characters of representations from $K$-theory. Finally, we obtain a continuity property near the identity element of $G$ of families of maps $K_0(C^*_rG)\to \C$, parametrised by semisimple elements of $G$, defined by stable orbital integrals. This  implies a continuity property for $L$-packets of discrete series characters, which in turn can be used to deduce a (well-known) expression for formal degrees of discrete series representations from Harish-Chandra's character formula.
\end{abstract}

\tableofcontents

\section{Introduction}

Let $G$ be a real semismple Lie group. Its reduced $C^*$-algebra $C^*_rG$ is the closure in $\cB(L^2(G))$ of the algebra of convolution operators by functions in $L^1(G)$. It represents the tempered dual of $G$ as a `noncommutative space' in the sense of noncommutative geometry, and
encodes all tempered representations of $G$. Its $K$-theory $K_*(C^*_rG)$ is a natural invariant to consider. This $K$-theory is described explicitly in terms of equivariant indices of Dirac operators on $G/K$, for a maximal compact subgroup $K<G$, in the Connes--Kasparov conjecture. This was proved in various cases by Penington and Plymen \cite{Penington83}, Wassermann \cite{Wassermann87}, Lafforgue \cite{Lafforgue02b} and finally in general by Chabert, Echterhoff and Nest \cite{CEN}.

Despite this explicit knowledge about the structure of $K_*(C^*_rG)$, it remains a challenge to extract explicit representation theoretic information from this $K$-theory group. There has been a good amount of success in this direction for classes in $K_*(C^*_rG)$ corresponding to discrete series representations, for groups having such representations.
For example, Lafforgue \cite{Lafforgue02} used $K$-theory to recover Harish-Chandra's criterion $\rank(G) = \rank(K)$ for the existence of discrete series representations.

The von Neumann trace $\tau_e$ on $C^*_rG$, defined by $\tau_e(f) = f(e)$ for $f$ in a dense subalgebra, induces a map on $K_0(C^*_rG)$. On classes corresponding to the discrete series, this gives the formal degrees of such representations. But this trace maps all other classes to zero (see Proposition~7.3 in \cite{Connes82}). It has recently become clear that a natural generalisation of the von Neumann trace involving \emph{orbital integrals} can be used to extract much more information from $K_0(C^*_rG)$. For a semisimple element $g \in G$, the orbital integral $\tau_g(f)$ of a function $f$ on $G$ is the integral of $f$ over the conjugacy class of $g$. This integral converges for $f$ in  Harish-Chandra's Schwartz algebra, which has the same $K$-theory as $C^*_rG$. That leads to maps 
\beq{eq taug}
\tau_g\colon K_0(C^*_rG) \to \C.
\eeq

If $D$ is an elliptic operator on a $\Z_2$-graded vector bundle over a manifold $M$, $G$-equivariant for a proper, cocompact action by $G$ on $M$, then one has the equivariant index
\[
\ind_G(D) \in K_0(C^*_rG).
\]
In \cite{HW2}, the authors proved a fixed point formula for the numbers
\beq{eq taug ind}
\tau_g(\ind_G(D)).
\eeq
They showed that Harish-Chandra's character formula for the discrete series is a special case of this fixed point formula, much as Weyl's character formula is a special case of the Atiyah--Segal--Singer \cite{Atiyah68} or Atiyah--Bott \cite{ABI} fixed point formulas, as proved in \cite{ABII}. Also, Shelstad's character identities for $L$-packets of representations follows from a $K$-theoretic argument involving $\tau_g$, in the case of discrete series representations \cite{HW3}.

For discrete groups, orbital integrals (now sums over conjugacy classes) are also useful tools in $K$-theory. The main result in \cite{Wangwang} is a fixed point theorem for \eqref{eq taug ind} in the discrete group case, which has consequences to orbifold geometry, positive scalar curvature metrics, and trace formulas. Gong \cite{Gong} and Samurka\c{s} \cite{Samurkas} used such maps on the $K$-theory of maximal group $C^*$-algebras to deduce information about rigidity of manifolds. And Xie and Yu have an article in preparation about an APS-type index theorem involving $\tau_g$.

For semisimple Lie groups $G$, the results in \cite{HW2, HW3, Lafforgue02b}  mentioned above show that classes in $K_0(C^*_rG)$ corresponding to the discrete series contain a great deal of information about those representations. But it was long unclear what (representation theoretic) information can be recovered from other classes. That question was important motivation for this paper. As a concrete  example, it was not known what information the generator of $K_0(C^*_r\SL(2,\R))$ corresponding to the limits of discrete series (or to the non-spherical principal series) contains.

In the present paper, we investigate further properties and applications of the maps \eqref{eq taug} for semisimple Lie groups, many of them related to the fixed point formula for \eqref{eq taug ind}. This starts with an explicit expression for $\tau_g$ applied to $K$-theory generators defined via Dirac induction (Theorem \ref{thm fixed GK}). That result shows that $\tau_g$ is the zero map on $K$-theory if $\rank(G)\not=\rank(K)$, but it has interesting consequences if $\rank(G)=\rank(K)$.
These
 include
\begin{itemize}
\item a way to use the maps $\tau_g$ to distinguish elements of $K_0(C^*_rG)$ (Corollary \ref{cor sep pts});
\item an embedding of $K_0(C^*_rG)$ into the spaces of distributions on $G^{\reg}$ or $G$ (Corollary \ref{cor distr});
\item an induction formula from $K$-equivariant indices to $G$-equivariant ones (Corollary \ref{cor induction});
\item versions of Selberg's vanishing principle for classes in $K_0(C^*_rG)$ (Corollary \ref{cor selberg}) and matrix coefficients of the discrete series (Corollary \ref{cor selberg ds});
\item  a Tannaka-type duality result (Corollary \ref{cor reconstruct});
\item a result relating the value of $\tau_g$ on $K$-theory generators to characters of representations (Corollary \ref{cor char}).
\end{itemize}
Furthermore, Dirac induction is known to be injective (indeed, bijective), but we recover this injectivity independently as well.

In the last bullet point above, Corollary \ref{cor char} explicitly states that $\tau_g$ maps a $K$-theory class to the value at $g$ of the character of one of the irreducible direct summands of the representation it corresponds to naturally. The values at $g$ of these characters are equal up to a sign, and they add up to zero if that representation is reducible. So the value at $g$ of one of these characters is the most relevant information one could have expected to obtain by applying $\tau_g$. 
This, to a large extent, answers the question if and what representation theoretic information is contained in classes in $K_0(C^*_rG)$ if $\rank(G)=\rank(K)$, even those not corresponding to the discrete series.
 In particular, the generator of $K_0(C^*_r\SL(2,\R))$ corresponding to the limits of discrete series
determines the characters of these representations on $K$.

For a fixed element $x \in K_0(C^*_rG)$, we will see that $\tau_g(x)$ does not depend continuously on $g$, for example at the identity element $e$. Theorem \ref{thm cts} states that a modified version of $\tau_g$, related to $L$-packets of representations in the Langlands program, has better continuity properties at $e$. That implies continuity of certain finite sums of discrete series characters (Corollary \ref{cor char cts}). And that can be used to take the limit as $g \to e$ in Harish-Chandra's character formula for the discrete series to obtain expressions for formal degrees of discrete series representations.


We hope that the various applications of orbital integrals to $K$-theory of group $C^*$-algebras in this paper
 help to demonstrate the relevance of orbital integrals as a tool to study such $K$-theory groups. In future work, we hope to generalise the results and their applications in this paper to more general groups.
  
\subsection*{Acknowledgements}

The authors thank Yanli Song for useful discussions.
The second author was supported by the Australian Research Council, through Discovery Early Career
Researcher Award DE160100525.

\section{Preliminaries}

Throughout this paper, let $G$ be a connected semisimple Lie group with finite centre. Let $K<G$ be a maximal compact subgroup. 
For any Lie group, we will denote its Lie algebra by the corresponding gothic letter.
Fix a $K$-invariant inner product on $\kg$, and let $\kp \subset \kg$ be the orthogonal complement to $\kk$. Then  $\kg = \kk \oplus \kp$. 

\subsection{Dirac induction}

The map $\Ad\colon K \to \SO(\kp)$ lifts to $\widetilde{\Ad}\colon \tilK \to \Spin(\kp)$, for a double cover $\widetilde {K}$  of $K$. 
Let $\Delta_{\kp}$ be the standard representation of $\Spin(\kp)$, viewed as a representation of $\tilK$ via $\widetilde{\Ad}$. 
Let $\hat K_{\Spin}$ be the set of 
 irreducible representations $V$ of $\tilK$ such that $\Delta_{\kp} \otimes V$ descends to a representation of $K$. Let $R_{\Spin}(K)$ be the free abelian group generated by $\hat K_{\Spin}$. 
 
 Let $V \in \hat K_{\Spin}$. Then we have the $G$-equivariant vector bundle
\[
E_V := G\times_K(\Delta_{\kp} \otimes V) \to G/K.
\]
Let $\{X_1, \ldots, X_{\dim(G/K)}\}$ be an orthonormal basis of $\kp$. Let $c_{\kp}\colon \kp \to \End(\Delta_{\kp})$ be the Clifford action. Let $L\colon \kg \to \End(C^{\infty}(G))$ be the infinitesimal left regular representation. Consider the Dirac operator
\[
D_V := \sum_{j=1}^{\dim(G/K)} L(X_j) \otimes c_{\kp}(X_j) \otimes 1_V
\]
on
\[
\Gamma^{\infty}(E_V) = \bigl( C^{\infty}(G) \otimes \Delta_{\kp} \otimes V\bigr)^K.
\]

If $G/K$ has a $G$-invariant $\Spin$-structure (which is the case precisely if $\Delta_{\kp}$ descends to $K$), then $D_V$ is the $\Spin$-Dirac operator on $G/K$ coupled to the bundle $G \times_K V \to G/K$, see Proposition 1.1 in \cite{Parthasarathy72}.
In any case, $D_V$ is a $G$-equivariant elliptic differential operator, and has an index
\[
\ind_G(D_V) \in K_*(C^*_rG).
\]
Here $C^*_rG$ is the reduced group $C^*$-algebra of $G$, and $\ind_G$ is the analytic assembly map \cite{Connes94}.  If $\dim(G/K)$ is even, then $\Delta_{\kp}$, and hence $E_V$, has a natural $\Z_2$-grading with respect to which $D_V$ is odd. Then $\ind_G(D_V) \in K_0(C^*_rG)$. If $\dim(G/K)$ is odd, then there is no such grading, and $\ind_G(D_V) \in K_1(C^*_rG)$. So in general, we have
\[
\ind_G(D_V) \in K_{\dim(G/K)}(C^*_rG).
\]
\emph{Dirac induction} is the map
\[
\DInd_K^G\colon R_{\Spin}(K) \to K_{\dim(G/K)}(C^*_rG)
\]
given by 
\[
\DInd_K^G[V] = \ind_G(D_V),
\]
with $V$ as above. By the Connes--Kasparov conjecture, proved in \cite{CEN, Lafforgue02b, Wassermann87}, this map is an isomorphism of abelian groups.

From now on, we suppose that $G/K$ is even-dimensional, since the $K$-theory group $K_0(C^*_rG)$ we study is zero otherwise.

\subsection{Orbital integrals and a fixed point formula}

Let $g \in G$ be a semisimple element. Let $Z_G(g)<G$ be its centraliser. Let $d(hZ_G(g))$ be the left invariant measure on $G/Z_G(g)$ determined by a Haar measure $dg$ on $G$. The \emph{orbital integral} with respect to $g$ of a measurable function $f$ on $G$ is
\[
\tau_g(f) := \int_{G/Z_G(g)} f(hgh^{-1})\, d(hZ_G(g)),
\]
if the integral converges. Harish-Chandra proved that the integral converges for $f$ in the Harish-Chandra Schwartz algebra $\cC(G)$, see Theorem 6 in \cite{HC66}. The subalgebra $\cC(G) \subset C^*_rG$ is dense and closed under holomorphic functional calculus (see Theorem 2.3 in \cite{HW2}). Hence we obtain a map
\[
\tau_g\colon K_0(C^*_rG) = K_0(\cC(G)) \to \C.
\]
Note that $\tau_e$ is the usual von Neumann trace.


Let $M$ be a Riemannian manifold with a proper, isometric, cocompact action by $G$. Let $E \to M$ be a $G$-equivariant, Hermitian, $\Z_2$-graded vector bundle. Let $D$ be an odd, self-adjoint, $G$-equivariant, elliptic differential operator on $E$. Then we have
\[
\ind_G(D) \in K_0(C^*_rG).
\]
In \cite{HW2}, the authors proved a fixed-point formula for the number $\tau_g(\ind_G(D))$, for almost all $g \in G$. Consequences include Harish-Chandra's character formula for the discrete series (Theorem 16 in \cite{HC66}; see Corollary 2.6 in \cite{HW2}) and Shelstad's character identities in the case of discrete series representations (\cite{Shelstad79}; see Theorem 2.5 in \cite{HW3}). In this paper, we explore further consequences.


To state the fixed point formula in \cite{HW2}, let
$\cN \to M^g$ be the normal bundle to the fixed point set $M^g$ of $g$ in $M$. Let $\sigma_D$ be the principal symbol of $D$. Let $c^g \in C_c(M^g)$ be nonnegative, and  such that for all $m \in M^g$,
\[
\int_{Z_G(g)} c^g(hm)\, dh = 1,
\]
for a fixed Haar measure $dh$ on $Z_G(g)$ compatible with $dg$ and $d(hZ_G(g))$. If $G/K$ is odd-dimensional, then $K_0(C^*_rG)=0$, so $\tau_g(\ind_G(D))=0$.
\begin{theorem}\label{thm fixed pt}
If $G/K$ is even-dimensional, then for almost all semisimple $g \in G$, we have $\tau_g(\ind_G(D))=0$ if $g$ is not contained in any compact subgroup of $G$, and 
\beq{eq fixed pt}
\tau_g(\ind_G(D)) = \int_{TM^g}c^g \frac{\ch\bigl([\sigma_D|_{\supp(c^g)}](g)\bigr)\Todd(TM^g \otimes \C)}{\ch\bigl([\Bigwedge \cN \otimes \C] (g) \bigr)}
\eeq
if it is.
\end{theorem}
Here $\ch\colon K^0(\supp(c^g)) \to H^*(\supp(c^g))$ and $\ch\colon K^0(TM^g|_{\supp(c^g)}) \to H^*(TM^g|_{\supp(c^g))}$ are Chern characters, and $\Todd$ denotes the Todd class. 

\begin{remark}\label{rem FGOI}
Explicitly, Theorem \ref{thm fixed pt} holds for the  semisimple $g \in G$ with \emph{finite Gaussian orbital integral} (FGOI), see  Definition 7 in  \cite{HW2}. That condition means that the integral
\[
\int_{G/Z_G(g)} e^{-d(e,hgh^{-1})^2}\, d(hZ_G(g))
\]
converges, where $d$ is the $G$-invariant Riemannian distance on $G$. It was shown in Proposition 4.2 in \cite{HW2} that almost every element of $G$ has FGOI.

In this paper, whenever a result is stated for almost all $g$, what is meant is that it holds for semisimple elements with FGOI, and possibly also with dense powers in a maximal torus.
\end{remark}


\section{A fixed point formula on $G/K$}

Let $T<K$ be a maximal torus. Let $\tilT <\tilK$ be its inverse image in $\tilK$.
Fix a set $R^+_c$ of positive roots of $(\kk^{\C},\kt^{\C})$. Let $\rho_c$ be half the sum of the elements of $R_c^+$.
Let $V \in \hat K_{\Spin}$. 
Let $\lambda \in i\kt^*$ be its highest weight with respect to $R_c^+$. 

For any finite-dimensional (actual or virtual) representation $W$ of $K$ or $\tilK$, we denote its character by $\chi_W$. For any function $\varphi$ on $\widetilde K$ that descends to a function on $K$, we will use the same notation $\varphi$ for both the function on $\tilK$ and $K$. E.g., we have $\chi_{\Delta_{\kp}}\chi_V \in C^{\infty}(K)$.

In the case where $T$ is a Cartan subgroup of $G$, i.e.\ $\rank(G)=\rank(K)$, fix a set of positive noncompact roots $R_n^+$ of $(\kg^{\C}, \kt^{\C})$ such that the character $\chi_{\Delta_{\kp}}$ of the graded representation $\Delta_{\kp}$ of $\tilde K$ satisfies
\beq{eq char Delta p}
\chi_{\Delta_{\kp}}|_{\tilde T} = \prod_{\alpha \in R_n^+} (e^{\alpha/2}-e^{-\alpha/2}).
\eeq
Such a choice of positive noncompact roots can always be made, see for example pages 17 and 18 of \cite{Atiyah77},  Remark 2.2 in \cite{Parthasarathy72} and (5.1) in \cite{ASIII}. In the equal-rank case, we write $R^+ := R_c^+ \cup R_n^+$. We will denote half the sums of the elements of $R^+$ and $R_n^+$ by $\rho$ and $\rho_n$, respectively.

Let $W_K := N_K(T)/T$ be the Weyl group of $(K,T)$.

\begin{theorem} \label{thm fixed GK}
\begin{itemize}
\item[(a)]
If $\rank(G) = \rank(K)$, then 
for almost all $g \in T$,
\[
\begin{split}
\tau_g(\DInd_K^G[V]) 
&= (-1)^{\dim(G/K)/2} \frac{\chi_V}{\chi_{\Delta_{\kp}}}(g)\\
&=  (-1)^{\dim(G/K)/2}\frac{\sum_{w \in W_K} \varepsilon(w)e^{w(\lambda+ \rho_c)}}{\prod_{\alpha \in R^+} (e^{\alpha/2}-e^{-\alpha/2})}(g).
\end{split}
\]
(In particular, the right hand sides are well-defined.)
\item[(b)]
If $\rank(G) \not= \rank(K)$, then 
for almost all $g \in T$,
\[
\tau_g(\DInd_K^G[V]) =  0.
\]
\end{itemize}
\end{theorem}

Let $\ka \subset \kp$ be an abelian subspace such that $Z_{\kg}(\kt) = \kt \oplus \ka$. Let $c \in C_c(\ka)$ be a function whose integral over $\ka$ is $1$. Let $\sigma_{D_V}$ be the principal symbol of $D_V$.
\begin{lemma} \label{lem fixed pt 1}
For almost all $g \in T$,
\[
\tau_g(\DInd_K^G[V]) = 
\int_{T\ka}c \frac{\ch\bigl([\sigma_{D_V}|_{\supp(c)}](g)\bigr)}{\ch\bigl([\ka \times \Bigwedge \kp/\ka\otimes \C] (g) \bigr)}.
\]
\end{lemma}
\begin{proof}
Let $g \in T$ be such that its powers are dense in $T$, and with FGOI (see Remark \ref{rem FGOI}). 
 By Proposition 4.2 in \cite{HW2}, almost all elements of $T$ have these two properties.

We have $G/K \cong \kp$ as $K$-spaces, hence in particular as $T$-spaces. Hence
\[
(G/K)^g = (G/K)^T = \kp^{\Ad(T)} = \ka.
\]
Set $A := \exp(\ka)$; this is the centraliser of $g$ in $\exp(\kp)$. 
We have $\kp = \ka \oplus \kp/\ka$ as representations of $T$. So
the normal bundle in $G/K = \kp$ to $(G/K)^g = \ka$ is $\ka \times \kp/\ka \to \ka$. 
The Todd class of the trivial bundle $T(G/K)^g \otimes \C \to (G/K)^g$ is $1$. Hence the claim follows from 
 Theorem \ref{thm fixed pt}. 
\end{proof}

Let us compute $[\sigma_{D_V}|_{\supp(c)}]$. Let $\beta_{\ka} \in K^0(\ka)$ be the Bott generator. (Note that $\ka$ is even-dimensional since $G/K$ is.) Let $\pi\colon T\ka \to \ka$ be the tangent bundle projection, and $\pi|_{\supp(c)}\colon \supp(c) \times \ka\to \supp(c)$ its restriction. Note that
\[
\Delta_{\kp} \cong \Delta_{\ka} \otimes \Delta_{\kp/\ka}
\]
as graded representations of $\tilT$. These descend to $T$ after tensoring with $V$.

\begin{lemma}\label{lem sigma DV}
Under the isomorphism
\[
K_0^T(\supp(c) \times \ka) \cong K_0(\supp(c) \times \ka) \otimes R(T),
\]
we have
\[
[\sigma_{D_V}|_{\supp(c)}] \mapsto \pi|_{\supp(c)}^*\beta_{\ka} \otimes [\Delta_{\kp/\ka} \otimes V].
\]
\end{lemma}
\begin{proof}
Let $c_{\ka} \colon \ka \to \End(\Delta_{\ka})$ be the Clifford action. The class  $\pi|_{\supp(c)}^*\beta_{\ka} \in K^0(\supp(c) \times \ka)$ is defined by\footnote{We absorb a possible sign in the definition of $\beta_{\ka}$; see Lemma 4.1 in \cite{Connes82}.} the vector bundle homomorphism
\[
A\colon \supp(c) \times \Delta_{\ka}^+ \to \supp(c) \times \Delta_{\ka}^- 
\]
given by
\[
A_Y = c_{\ka}(Y)
\]
for all $Y \in \supp(c)$.

We have
\[
(G\times_K (\Delta^{\pm}_{\kp} \otimes V))|_\ka \cong \ka \times  \Delta^{\pm}_{\kp} \otimes V
\]
as $T$-vector bundles. So
\[
\pi|_{\supp(c)}^*\bigl((G\times_K (\Delta^{\pm}_{\kp} \otimes V))|_{\supp(c)} \bigr) = (\supp(c) \times \ka) \times \Delta^{\pm}_{\kp} \otimes V.
\]
Let $X,Y \in \ka$, so that, using the above identification, we get 
\beq{eq sigma DV}
\sigma_{D_V}(X,Y) = c_{\kp}(Y) \otimes 1_V\colon \Delta^{+}_{\kp} \otimes V \to \Delta^{-}_{\kp} \otimes V.
\eeq

Since $Y \in \ka$, the map \eqref{eq sigma DV} equals the odd endomorphism
\[
c_{\ka}(Y) \otimes 1_{\Delta_{\kp/\ka} \otimes V} \in \End(\Delta_{\ka} \otimes \Delta_{\kp/\ka} \otimes V)
\]
Together with the above form of the class $\pi|_{\supp(c)}^*\beta_{\ka}$, this implies the claim.
\end{proof}

\begin{lemma} \label{lem wedge p}
Suppose that $\rank(G) = \rank(K)$. Then
\[
\Bigwedge \kp\otimes \C =  (-1)^{\dim(G/K)/2}\Delta_{\kp}\otimes  \Delta_{\kp}
\]
as graded representations of $T$.
\end{lemma}
\begin{proof}
The set of positive noncompact roots $R_n^+$ determines a complex structure on $\kp$ such that $\kp^{1,0}$ is the sum of the positive noncompact root systems.
As graded representations of $T$, we have
\[
\Bigwedge \kp\otimes \C = \Bigwedge \kp^{1,0} \otimes \Bigwedge \kp^{0,1} = \Bigwedge_{\C} \kp \otimes (\Bigwedge_{\C} \kp)^*. 
\]
The element $\rho_n \in i\kt^*$ is integral for $\tilT$, and  $\Delta_{\kp} \otimes \C_{\rho_n}$ descends to a representation of $T$. We have
\[
\Bigwedge_{\C} \kp = (-1)^{\dim(G/K)/2} \Delta_{\kp} \otimes\C_{\rho_n}
\]
as graded representations of $T$; see for example the proof of Lemma 5.5 in \cite{HW2}. Since $\Delta_{\kp}^* \cong (-1)^{\dim(G/K)/2} \Delta_{\kp}$, we conclude that
\[
\Bigwedge \kp\otimes \C = \Delta_{\kp} \otimes \Delta_{\kp}^* = (-1)^{\dim(G/K)/2}\Delta_{\kp}\otimes  \Delta_{\kp}.
\]
The nontrivial element of the kernel of the covering map $\tilK \to K$ acts on $\Delta_{\kp}$ as $\pm 1$; therefore $\Delta_{\kp}\otimes  \Delta_{\kp}$ descends to a representation of $T$.
\end{proof}

\begin{lemma} \label{lem Dirac Rn}
Let $c$ be a nonnegative, compactly supported, continuous function on $\R^{2n}$ with integral $1$. Let $\beta \in K^0(\R^{2n})$ be the Bott class, and $\pi|_{\supp(c)} \colon \supp(c) \times \R^{2n} \to \supp(c)$ where $\pi: T\R^{2n}\rightarrow\R^{2n}$ the natural projection. Then
\beq{eq int Bott}
\int_{\R^{2n}\times \R^{2n}} c\ch(\pi|_{\supp(c)}^*\beta) = 0.
\eeq
\end{lemma}
\begin{proof}
By Proposition 6.11 in \cite{Wang14}, the integral \eqref{eq int Bott} equals the $L^2$-index of the $\Spin$-Dirac operator on $\R^{2n}$. That index is zero because the $L^2$-kernel of this Dirac operator is zero. Indeed, the $\Spin$-Dirac operator on $\R^{2n}$ only has continuous spectrum, see for example Theorem 7.2.1 in \cite{Ginoux76}.
\end{proof}

\begin{proof}[Proof of Theorem \ref{thm fixed GK}]
Lemma \ref{lem sigma DV} implies that
\[
\ch\bigl([\sigma_{D_V}|_{\supp(c)}](g)\bigr)= \ch(\pi|_{\supp(c)}^*\beta_{\ka}) (\chi_{\Delta_{\kp/\ka}}\chi_V)(g).
\]
Furthermore, 
\[
\ch\bigl([\ka \times \Bigwedge \kp/\ka\otimes \C] (g) \bigr)= \chi_{\bigwedge \kp/\ka \otimes \C}(g)
\]
in the graded sense. So by Lemma \ref{lem fixed pt 1},
\[
\tau_g(\DInd_K^G[V]) = 
 \frac{\chi_{\Delta_{\kp/\ka}}\chi_V}{\chi_{\bigwedge \kp/\ka \otimes \C}}(g)
\int_{T\ka}c \ch(\pi|_{\supp(c)}^*\beta_{\ka}).
\]

If $\rank(G) \not=\rank(K)$, then $\ka$ is nonzero, and the claim follows from Lemma \ref{lem Dirac Rn}. 
If $\rank(G) =\rank(K)$, then  Lemma \ref{lem wedge p} implies that
\[
\tau_g(\DInd_K^G[V]) =  (-1)^{\dim(G/K)/2}\frac{\chi_V}{\chi_{\Delta_{\kp}}}(g);
\]
in particular, the right hand side is well-defined.
The claim now follows from Weyl's character formula and \eqref{eq char Delta p}. (Note that $(\tilK,\tilT)$ and $(K,T)$ have the same Weyl group $W_K$, since they have the same root system.)
\end{proof}

\begin{remark}
If $g=e$, then $\tau_e(\DInd_K^G[V])$ is the $L^2$-index of $D_V$ by Proposition 4.4 in \cite{Wang14}. That index is zero if the kernel of $D_V$ is zero. Theorem \ref{thm fixed GK} shows that,  in the equal-rank case, the more general trace $\tau_g$ yields nonzero information even in cases where the kernel of $D_V$ is zero (see also Section~\ref{sec SL2}). 
\end{remark}

\section{Consequences}

Suppose from now on that $\rank(G) = \rank(K)$.

\subsection{Distinguishing $K$-theory classes}

As a consequence of Theorem \ref{thm fixed GK}, the traces $\tau_g$ `separate points' on $K_0(C^*_rG)$, or distinguish all elements of $K_0(C^*_rG)$, in the following sense.
\begin{corollary}\label{cor sep pts}
Let $x\in K_0(C^*_rG)$. If $\tau_g(x)=0$ for all $g$ in a dense subset of $T$, then $x=0$.
\end{corollary}
\begin{proof}
Let $x\in K_0(C^*_rG)$. 
By surjectivity of Dirac induction, we can write
\[
x = \sum_{V\in \hat K_{\Spin}} m_{V}\DInd_K^G[V],
\]
for $m_{V} \in \Z$, finitely many nonzero.
By Theorem \ref{thm fixed GK}, we have for almost all $g\in T$,
\[
\tau_g(x) = (-1)^{\dim(G/K)/2}\sum_{V\in \hat K_{\Spin}} m_{V}\frac{\chi_{V}}{\chi_{\Delta_{\kp}}}(g).
\]
So if $\tau_g(x)=0$ for all $g$ in a dense subset of $T$, then 
by continuity and conjugation invariance of the characters $\chi_{V}$, we find that 
%
%
\[
\sum_{V\in \hat K_{\Spin}} m_{V}{\chi_{V}} = 0.
\]
So $m_{V} = 0$ for all $V$, i.e.\ $x=0$.
\end{proof}

\subsection{$K$-theory and distributions}

Let $G^{\reg} \subset G$ be the subset of regular elements.

\begin{corollary}\label{cor distr}
The map
\[
\tau\colon K_0(C^*_rG) \to \calD'(G^{\reg})
\]
defined by
\[
\langle \tau(x), f\rangle =\int_{G^{\reg}} \tau_g(x) f(g)\, dg
\]
for $x \in K_0(C^*_rG)$ and $f \in C^{\infty}_c(G^{\reg})$,
is a well-defined, injective group homomorphism.
\end{corollary}
\begin{proof}
Let $x \in K_0(C^*_rG)$. By surjectivity of Dirac induction, we can write $x = \DInd_K^G[y]$, for some $y \in R_{\Spin}(K)$. 
Theorem \ref{thm fixed GK} implies that the function $g\mapsto \tau_g(x)$  equals an analytic function almost everywhere on the set of elliptic elements of $G$. Theorem \ref{thm fixed pt} implies that this function equals zero almost everywhere on the set of non-elliptic elements of $G$. So $g\mapsto \tau_g(x)$
equals an analytic function almost everywhere on $G$. Furthermore, that analytic function is bounded on compact subsets of $G^{\reg}$. This implies that $\tau(x)$ is a well-defined distribution on $G^{\reg}$.

If $\tau(x) = 0$, then $\tau_g(x)=0$ for almost all $g \in G^{\reg}$, in particular for almost all elements of $T$. Hence Corollary \ref{cor sep pts} implies that $x=0$.
\end{proof}

\begin{remark}
As noted in the proof of Corollary \ref{cor distr}, the first part of Theorem \ref{thm fixed pt} implies that $\tau(x)$ is zero outside the set of regular elliptic elements of $G$. 
\end{remark}

\begin{remark}
We will describe the map $\tau$ in Corollary \ref{cor distr} explicitly in terms of characters of representations in Section \ref{sec lds}. 
There we will see that $\tau(x)$ equals the character of a tempered representation of $G$ almost everywhere on the set of regular elliptic elements, and zero almost everywhere outside the set of elliptic elements. Therefore, it extends to a distribution on all of $G$ by Harish-Chandra's regularity theorem.
\end{remark}


\subsection{Injectivity of Dirac induction}

We have used the surjectivity of Dirac induction in the proof of Corollary \ref{cor sep pts} (which is justified because the Connes--Kasparov conjecture has been proved). Theorem \ref{thm fixed GK}
 implies injectivity of Dirac induction.
\begin{corollary} \label{cor Dirac inj}
Dirac induction is injective.
\end{corollary}
\begin{proof}
Let $y\in R_{\Spin}(K)$, and suppose that $\DInd_K^G(y)=0$. Then $\tau_g(\DInd_K^G(y))=0$ for all $g\in T$. Theorem \ref{thm fixed GK} implies that for almost all $g \in T$,
\[
\frac{\chi_y}{\chi_{\Delta_{\kp}}}(g) = 0.
\]
So $\chi_y = 0$, i.e.\ $y=0$.
\end{proof}

\subsection{An induction formula}

Let $M$ be an even-dimensional Riemannian manifold with a $G$-equivariant $\Spinc$-structure. Let $E \to M$ be a $G$-equivariant, Hermitian vector bundle. Let $D_M^E$ be the $\Spinc$-Dirac operator on $M$ twisted by $E$. By Abels' theorem \cite{Abels}, there is a $K$-invariant submanifold $N \subset M$ such that $M \cong G\times_N N$ via the action map $G \times N \to M$. Furthermore, $N$ has a $K$-equivariant $\Spinc$-structure on $N$ compatible with the one on $M$, see Proposition 3.10 in \cite{HochsMathai17}. The $\Spinc$-Dirac operator $D_N^E$ on $N$, twisted by $E|_N$, has the property that
\beq{eq Q Ind}
\DInd_K^G(\ind_K(D_N^E)) = \ind_G(D_M^G) \quad \in K_0(C^*_rG).
\eeq
See Theorem 5.2 in \cite{HW2} and Proposition 4.7 in \cite{Hochs09}.

Theorem \ref{thm fixed GK} and surjectivity of Dirac induction imply that the following diagram commutes for all $g$ in the dense subset of $T$ in Theorem \ref{thm fixed GK}:
\beq{eq taug Ind}
\xymatrix{
K_0(C^*_rG) \ar[drrr]^-{ \tau_g }  & & & \\
R_{\Spin}(K) \ar[u]^-{\DInd_K^G} \ar[rrr]_-{\vspace{-3mm}(-1)^{\dim(G/K)/2}\ev_g/\chi_{\Delta_{\kp}}(g)} & & & \C.
}
\eeq
Here $\ev_g$ denotes evaluation of characters of representations  at $g$; note that the bottom arrow is well-defined.

The equality \eqref{eq Q Ind} and commutativity of \eqref{eq taug Ind} imply the following formula for induction from slices.
\begin{corollary} \label{cor induction}
We have, for almost all $g \in T$,
\[
\tau_g(\ind_G(D_M^E)) = (-1)^{\dim(G/K)/2}\ind_K(D_N^E)(g)/\chi_{\Delta_{\kp}}(g).
\]
\end{corollary}
Note that the right hand side can be computed via the Atiyah--Segal--Singer fixed point formula \cite{Atiyah68}. 

Induction formulas like Corollary \ref{cor induction} we used in various settings to deduce results about $G$-equivariant indices from results about $K$-equivariant indices \cite{GMW, Hochs09, HM16, HochsMathai17, HW2}. The case $g=e$ is not covered by Corollary \ref{cor induction}; that case is Corollary 53 in \cite{GMW}.

\subsection{Selberg's principle}

The Selberg principle is a vanishing result for orbital integrals of certain convolution idempotents on $G$. See \cite{Blanc92, JulgValette86, JulgValette87} for approaches to this principle in the spirit of noncommutative geometry. Theorem \ref{thm fixed pt}
implies a version of this principle.
\begin{corollary}[$K$-theoretic Selberg principle] \label{cor selberg}
For almost all $g$ not contained in compact subgroups of $G$, the map
\[
\tau_g \colon K_0(C^*_rG) \to \C
\]
is zero.
\end{corollary}
\begin{proof}
Theorem \ref{thm fixed pt} implies that for almost all $g$ not contained in compact subgroups of $G$, and all $V \in R_{\Spin}(K)$, we have
\[
\tau_g(\DInd_K^G[V]) = 0.
\]
So surjectivity of Dirac induction implies the claim.
\end{proof}

Corollary \ref{cor selberg} has a purely representation theoretic consequence.
\begin{corollary}[Selberg principle for matrix coefficients of the discrete series] \label{cor selberg ds}
Let $\pi$ be a discrete series representation of $G$. 
Let $v$ be a $K$-finite vector in the representation space of $\pi$, and $m_{v,v}$ the corresponding matrix coefficient. For all $g$ not contained in compact subgroups of $G$, we have
\[
\tau_g(m_{v,v})=0.
\]
\end{corollary}
\begin{proof}
Let $d_{\pi}$ be the formal degree of $\pi$. 
By rescaling, we may assume that $v$ has norm $1$.
Then ${d_{\pi}}\bar m_{v,v}$ is an idempotent in $C^*_rG$. Let $[\pi] \in K_0(C^*_rG)$ be its $K$-theory class.
Since $v$ is $K$-finite, the function $m_{v,v}$ lies in Harish-Chandra's Schwartz algebra $\cC(G)$. Therefore, for all semisimple $g \in G$,
\[
\tau_g(m_{v,v}) = \frac{1}{d_{\pi}} \overline{\tau_g([\pi])}.
\]
By Corollary \ref{cor selberg}, the number is zero for almost all $g$ not contained in compact subgroups. The claim therefore follows by continuity of $m_{v,v}$.
\end{proof}

\subsection{A Tannaka-type duality}

We now suppose that the representation $\Delta_{\kp}$ of $\tilK$ descends to $K$. This is true of we replace $G$ by a double cover if necessary. Then Dirac induction is defined on $R(K)$.

The $K$-theory group $K_0(C^*_rG)$ and its elements contain nontrivial information about $G$ and its representations, see e.g.\ \cite{HW2, HW3, Lafforgue02b}. But 
 just the isomorphism class of 
$K_0(C^*_rG)$ as an abelian group contains no information about $G$ whatsoever: this group is always
 free, with countably infinitely many generators.
It turns out, however,  that the combination of the isomorphism class of $K_0(C^*_rG)$,  the topological space $T$ and the maps $\tau_g\colon K_0(C^*_rG)\to \C$, for $g \in T$, together determine the Cartan motion group $K \ltimes \kp$ and vice versa. The tempered representation theory of $K\ltimes \kp$ is closely related to that of $G$; this is the Mackey analogy \cite{Afgoustidis, Higson08, Higson11, Mackey, 
 TianYaoYu, Yu17}. Also, the analytic assembly map for $G$ can be defined in terms of a continuous deformation from $K\ltimes \kp$ to $G$, see pp. 23--24 of \cite{Connes94} and \cite{Higson08}.

This  is vaguely analogous to the fact that the irrational rotation algebras $A_{\lambda}$, for irrational $\lambda$ in $[0,1/2]$, have the same $K$-theory $\Z \oplus \Z$, but are determined up to isomorphism by the pair $(K_0(A_{\lambda}), \tau)$ where $\tau$ is a natural trace. This is because the image of $\tau$ is $\Z + \lambda\Z$. 

\begin{corollary} \label{cor reconstruct}
The 
\begin{itemize}
\item abelian group $K_0(C^*_rG)$ up to isomorphism;
\item pointed topological space $(T, \{e\})$ up to homeomorphism; and
\item  family of group homomorphisms $\tau_g\colon K_0(C^*_rG)\to \C$, for $g \in T$
\end{itemize}
together determine
the Cartan motion group $K\ltimes \kp$, and vice versa. 
\end{corollary}
\begin{proof}
Write
\[
K_0(C^*_rG) = \bigoplus_{j \in \Z}\Z
\]
and let $e_j$ be a generator of the $j$th copy of $\Z$. 

Consider the function $\chi_j\colon T \to \R$ given by $\chi_j(g)=\tau_g(e_j)$. By Theorem \ref{thm fixed GK}, there is a function $\psi \in C^{\infty}(T)$, not unique but independent of $j$, and there are uniquely determined integers $d_j$ such that for all $j$, 
\[
\lim_{g\to e} \psi(g)\chi_j(g) = d_j,
\]
where at least one of the integers $d_j$ equals $1$. (Indeed, take $\psi = \chi_{\Delta_{\kp}}|_T$ and $d_j$ plus or minus the dimensions of the irreducible representations of $K$.) By replacing $e_j$ by $-e_j$ where necessary, we can make sure that all integers $d_j$ are positive.

Fix $j_0 \in \Z$ such that $d_{j_0} = 1$. Then, again by Theorem \ref{thm fixed GK},
\[
|\chi_{\Delta_{\kp}}|_T| = |\chi_{j_0}|^{-1}.
\]
And $\overline{\chi_{\Delta_{\kp}}} = -\chi_{\Delta_{\kp}}$, so $\chi_{\Delta_{\kp}}$ is imaginary-valued. Hence
\[
\chi_{\Delta_{\kp}}|_T = \pm  i|\chi_{j_0}|^{-1}.
\] 
We cannot resolve 
the sign ambiguity with the data we have, but we will not need to.

The characters of irreducible representations $V_j$ of $K$ are now determined by 
\[
\chi_{V_j}|_T = \chi_{\Delta_{\kp}}|_T \chi_j = \pm  i|\chi_{j_0}|^{-1} \chi_j,
\]
with the sign chosen such that $\pm  i|\chi_{j_0}|^{-1} \chi_j > 0$ near the identity element.  This determines the representations $V_j$ of $K$, and their tensor products and the underlying vector spaces. By Tannaka duality \cite{Tannaka}, this determines $K$.

To recover $\kp$ as a $K$-representation, set $\psi :=  i|\chi_{j_0}|^{-1}$. Then
\[
\psi = \pm
\chi_{\Delta_{\kp}}|_T = \pm \prod_{\alpha \in R^+_{n}}(e^{\alpha/2}-e^{-\alpha/2}). 
\]
This implies that for all $X,Y \in \kt$,
\[
\ddt \psi(X+tY) =\psi(\exp(X)) \sum_{\alpha \in R^+_{n}} \frac{\langle \alpha, Y \rangle}{2} \coth(\langle \alpha, X\rangle/2).
\]
The term on the right hand side corresponding to $\alpha$ equals the same term with $\alpha$ replaced by $-\alpha$. But otherwise this expression determines the weights $\alpha$ up to signs.  
In this way, we recover the set 
$R_{n}$ of $\kt$-weights of $\kp \otimes \C$ as a complex representation of $T$. And hence $\kp$ as a real representation of $T$, and therefore as a representation of $K$. This determines $K \ltimes \kp$.

Conversely, the Cartan motion group $K \ltimes \kp$ determines its maximal compact subgroup $K$ and the quotient $\kp = (K\ltimes \kp)/K$ as a representation of $K$. And $K$ determines the pair $(T, \{e\})$ up to conjugacy. The $K$-theory group $K_0(C^*_rG)$ is isomorphic to $R(K)$ via Dirac induction. Furthermore, $K$ and $\kp$
 determine the characters $\chi_V$, for $V \in \hat K$ and $\chi_{\Delta_{\kp}}$, and the dimension $\dim(G/K) = \dim(\kp)$. Hence, by Theorem \ref{thm fixed GK}, this determines the maps $\tau_g\colon K_0(C^*_rG) \cong R(K) \to \C$, for $g \in T$.
\end{proof}

\begin{remark}
In Corollary \ref{cor reconstruct}, $T$ may be replaced by a dense subset. Also, one only needs the neighbourhoods of the identity element, not all of its topology. And as stated in the corollary, one does not need the group structure of $T$.
\end{remark}

\begin{remark}
If $G=K$ is compact, then the triple $(K_0(C^*_rG), (T, \{e\}), (\tau_g)_{g\in T})$ determines the ring $R(G)$  of characters of $G$. That in turn determines the tensor products of representations of $G$, and forgetful maps to finite-dimensional complex vector spaces. So in this case, Corollary \ref{cor reconstruct} reduces to Tannaka duality for compact groups \cite{Tannaka} (which was used in the proof of Corollary \ref{cor reconstruct}).
\end{remark}

\begin{remark}
If the representation $\Delta_{\kp}$ of $\tilK$ does not descend to $K$, then we only recover the ring $R_{\Spin}(K)$ in the proof of Corollary \ref{cor reconstruct} and cannot directly apply Tannaka duality.
\end{remark}

\section{Characters} \label{sec lds}

Again, we suppose that the representation $\Delta_{\kp}$ of $\tilK$ descends to $K$. We may need to replace $G$ by a double cover for this assumption to hold. This assumption is now not essential; see Remark \ref{rem char}.

\subsection{Characters and $\tau_g$}

The structure of the $C^*$-algebra $C^*_rG$ and its $K$-theory was described by Wassermann \cite{Wassermann87} and Clare, Crisp and Higson \cite{CCH}. We can use this to relate values of $\tau_g$ on $K$-theory classes to values of characters of representations.

Let $P = MAN < G$ be a cuspidal parabolic and $\sigma$ in the set $\hat M_{\ds}$ of discrete series representations of $M$. Consider the bundle of Hilbert spaces $\cE_{P, \sigma} \to \hat A$ whose fibre at $\nu \in \hat A$ is $\Ind_P^G(\sigma \otimes \nu \otimes 1_N)$. (This can be topologised by viewing it as a trivial bundle in the compact picture of induced representations.) Let $\Ind_P^G(\sigma)$ be the Hilbert $C_0(\hat A)$-module of continuous sections of $\cE_{P, \sigma}$ vanishing at infinity. The group
\[
W_{\sigma} := \{w \in N_K(\ka)/Z_K(\ka); w\sigma = \sigma\}
\]
acts on $\cK(\Ind_P^G(\sigma))$ via Knapp--Stein intertwiners; see Theorem 6.1 in \cite{CCH}. Let $\cK(\Ind_P^G(\sigma))^{W_{\sigma}}$ be the fixed point algebra of this action. Then
\[
C^*_rG \cong \bigoplus_{P, \sigma} \cK(\Ind_P^G(\sigma))^{W_{\sigma}}
\]
where the sum runs over a set of cuspidal parabolics $P=MAN$ and $\sigma \in \hat M_{\ds}$. This is Theorem 6.8 in \cite{CCH}. See also Theorem 8 in \cite{Wassermann87}.

Now let $P$ and $\sigma$ be such that
\[
K_0(\cK(\Ind_P^G(\sigma))^{W_{\sigma}})
\]
is nonzero, hence infinite cyclic. (This is equivalent to the condition that $W_{\sigma}$ equals the $R$-group $R_{\sigma}$, see Lemma 10 in  \cite{Wassermann87}.)
Let $b(P, \sigma) \in K_0(C^*_rG)$ be the generator of this summand of $K_0(C^*_rG)$ in the image under Dirac induction of the $\Z_{\geq 1}$-span of $\hat K$ inside $R(K)$. 

Let $\eta \in i\kt_M^*$ be the Harish-Chandra parameter of $\sigma$, and $\tilde \eta \in i\kt^*$ its extension by zero on the orthogonal complement of $\kt_M$ in $\kt$. 
For any positive root system $\tilde R^+$ of $(\kg^{\C}, \kt^{\C})$ for which $\tilde \eta$ is dominant, let $\pi^G(\tilde \eta, \tilde R^+)$ be the corresponding (limit of) discrete series representation of $G$.
We need the following version of Schmid's character identities. This is  Lemma 12 in \cite{Wassermann87} in the equal rank case, but with information included about 
 the infinitesimal characters of the limits of discrete series representations that occur.
\begin{proposition}\label{prop Schmid id}
There are $2^{\dim(A)}$ choices of positive roots $R^+_1, \ldots, R^+_{2^{\dim(A)}} \subset R$, obtained from $R^+$ by the application of all combinations of $\dim(A)$ commuting reflections in simple noncompact roots, 
such that
\[
\Ind_P^G(\sigma \otimes 1_A \otimes 1_N) = \bigoplus_{j=1}^{2^{\dim(A)}} \pi^G(\tilde \eta, R^+_j).
\]
\end{proposition}
\begin{proof}
This is a special case of Theorem 13.3 in \cite{KZ2} for the maximal parabolic $G$ in the equal-rank group $G$.
\end{proof}

As before, let $\rho_c$ be half the sum of the compact positive roots. 
By Lemma 15(i) in \cite{Wassermann87}, the element $\tilde \eta - \rho_c$ is dominant for $K$. It is integral because $\Delta_{\kp}$ descends to $K$; this implies that $\rho_n$ and hence $\tilde \eta - \rho + \rho_n$ is integral.
\begin{proposition}[Wassermann] \label{prop DInd}
Let $V_{\tilde \eta - \rho_c} \in \hat K$ have highest weight $\tilde \eta - \rho_c$. Then
\[
\DInd_K^G[V_{\tilde \eta - \rho_c}] = b(P, \sigma).
\]
\end{proposition}
\begin{proof}
See the last page of \cite{Wassermann87}. This uses Proposition \ref{prop Schmid id}.
\end{proof}

Proposition \ref{prop Schmid id} and Harish-Chandra's character formula for (limits of) discrete series representations imply that the character of the representation $\Ind_P^G(\sigma \otimes 1_A \otimes 1_N)$ naturally associated to the $K$-theory generator $b(P, \sigma)$ is zero on $T$, if this representation is reducible. (See Subsection \ref{sec SL2} for an example.) Therefore, it is a useful property of the map $\tau_g$ that it maps $b(P, \sigma)$ to the possibly nonzero value of an irreducible summand of that representation.
\begin{corollary}\label{cor char}
For almost all  $g \in T$,
$\tau_g(b(P, \sigma))$
equals the value at $g$ of the character of one of the irreducible summands of $\Ind_P^G(\sigma \otimes 1_A \otimes 1_N)$. The values at $g$ of the characters of these summands at $g$ are all equal up to a sign.
\end{corollary}
\begin{proof}
Proposition \ref{prop DInd} and Theorem \ref{thm fixed GK} imply that
\[
\tau_g(b(P, \sigma)) = \tau_g(\DInd_K^G[V_{\tilde \eta - \rho_c}] )= 
(-1)^{\dim(G/K)/2}\frac{\sum_{w \in W_K} \varepsilon(w)e^{w\tilde \eta}}{\prod_{\alpha \in R^+} (e^{\alpha/2}-e^{-\alpha/2})}(g).
\]
By Harish-Chandra's character formula (extended coherently to the limits of discrete series), the right hand side is the value at $g$ of the character of $\pi^G(\tilde \eta, R^+)$. That formula also shows that on $T$, the character of $\pi^G(\tilde \eta, R^+)$ equals the character of $\pi^G(\tilde \eta, R_j^+)$ modulo a sign, for $j=1, \ldots, 2^{\dim(A)}$. Hence the claim follows from Proposition \ref{prop Schmid id}.
\end{proof}

\begin{remark} \label{rem char}
If the representation $\Delta_{\kp}$ does not descend to $K$, then the analogue of Corollary \ref{cor char} relates $\tau_g(b(P, \sigma))$ to characters of the corresponding representations of a double cover of $G$.
\end{remark}

\subsection{Example: non-spherical principal series and limits of discrete series of $\SL(2,\R)$}\label{sec SL2}

Consider the case where $G = \SL(2,\R)$, $K = T = \SO(2)$, and $P = MAN < \SL(2,\R)$ is the minimal parabolic of upper triangular matrices, where $M = \{\pm I\}$. Then $\hat M_{\ds} = \{\sigma_+, \sigma_-\}$, where $\sigma_+$ is the trivial representation of $M$ in $\C$ and $\sigma_-$ is the nontrivial one. Now we have Morita equivalences
\[
\begin{split}
\cK(\Ind_P^G(\sigma_+))^{W_{\sigma_+}}&\sim C_0([0,\infty));\\
\cK(\Ind_P^G(\sigma_-))^{W_{\sigma_-}}&\sim C_0(\R)\rtimes \Z_2.
\end{split}
\]
See Example 6.11 in \cite{CCH}.
So the pair $(P, \sigma_+)$ does not contribute to $K_0(C^*_r(\SL(2, \R)))$, whereas $(P, \sigma_-)$ contributes a summand $\Z$, generated by
\[
b(P, \sigma_-) = \DInd_K^G[\C_0].
\]

Let $\alpha \in i\kt^*$ be the root mapping $\begin{pmatrix} {0} & {-1} \\{1} &{0} \end{pmatrix} $ to $2i$. Set $R^+ := \{\alpha\}$. Let 
\[
g = \begin{pmatrix} \cos \varphi &-\sin \varphi \\ \sin \varphi  &\cos \varphi  \end{pmatrix} \in T,
\]
where $\varphi \in \R\setminus 2\pi \Q$. Theorem \ref{thm fixed GK} now yields
\[
\tau_g(b(P, \sigma_-)) = \frac{1}{2 i \sin \varphi}.
\]
This is the value at $g$ of the character of the limit of discrete series representation $\pi^G(0, R^+)$, and minus the value at $g$ of the character of the limit of discrete series representation $\pi^G(0, -R^+)$. The direct sum of these two representations is the non-spherical principal series representation $\Ind_P^G(\sigma_- \otimes 1_A \otimes 1_N)$. The character of that representation is zero at $g$.

Some authors, including the authors of this paper, have wondered if the $K$-theory generator $b(P, \sigma_-)$ can be detected by suitable maps out of $K_0(C^*_r(\SL(2,\R)))$, and if representation theoretic information can be recovered from it. This example shows that the answer to both questions is yes.

\section{Stable orbital integrals and continuity at the group identity}

This section is independent of the rest of this paper. In particular, it does not depend on Theorem \ref{thm fixed GK}.

It follows from Theorem \ref{thm fixed GK} that, for a fixed $x \in K_0(C^*_rG)$, the function
\[
g\mapsto \tau_g(x)
\]
on the set of semisimple elements $g$ of $G$, is not continuous if $G$ is noncompact. In particular, it is not continuous at the identity element.
 Theorem \ref{thm fixed GK} does imply that this function is continuous almost everywhere. Already in the compact case, it is a nontrivial question if the right hand side of the fixed point formula \eqref{eq fixed pt} depends continuously on $g$, for example as $g \to e$ (as pointed out in Section 8.1 in \cite{BGV}). It turns out that a version of $\tau_g$ involving \emph{stable orbital integrals} has better continuity properties near the identity element. (This comes at the cost of mapping more elements to zero, however. See Section \ref{sec SL2}, where the stable orbital integral of the class in $K_0(C^*_r\SL(2\,R))$ associated to the limits of discrete series is shown to be zero.)

\subsection{Continuity at $e$}

Let $G_{\C}$ be a complex semisimple Lie group, and $G<G_{\C}$ a real form of $G_{\C}$. Let $g$ be a semisimple element of $G$. 
\begin{definition}
The \emph{stable conjugacy class} of $g$ in $G$ is 
\[
(g)_s:=\{hgh^{-1}\in G : h\in G_{\C}\};
\]
the intersection of the conjugacy class $(g)_{G_{\C}}$ of $g$ in $G_{\C}$ with $G$. 

For every $f$ in the Harish-Chandra Schwartz algebra $\cC(G)$,
the \emph{stable orbital integral} of $f$ with respect to $g$ is
\[
\tau_g^s(f):=\sum_{g'}\tau_{g'}(f)=\sum_{g'}\int_{G/Z_G(g')}f(hg'h^{-1})dh(Z_G(g')),
\] 
where the sum is over representatives $g'$ of $G$-conjugacy classes in $(g)_s$, i.e., 
$(g)_s=\sqcup_{g'}(g').$
\end{definition}

Stable conjugacy classes are relevant to the notion of an $L$-packet of representations and Shelstad's character identities. See \cite{Shelstad79}.

The map $\tau_g^s\colon K_0(C^*_rG) = K_0(\cC(G)) \to \C$ induced by $\tau_g^s$, has better continuity properties in $g$ than $\tau_g$.
Let $S \subset G$ be the set of elements $g$ for which Theorem \ref{thm fixed pt} holds, see Remark \ref{rem FGOI}. Then $G\setminus S$ has measure zero, so in particular $S$ is dense.
\begin{theorem} \label{thm cts}
For all $x \in K_0(C^*_rG)$,
\[
\lim_{g\to e; g \in S} \tau_g^s(x) = \tau_e(x).
\]
\end{theorem}
(Note that $\tau_e = \tau_e^s$.)

Let $K<G$ be maximal compact. 
If $\rank(G)\not= \rank(K)$, then Theorem \ref{thm cts} follows from Theorem \ref{thm fixed GK}(b) and the fact that $\tau_e$ is identically zero on $K_0(C^*_rG)$. So assume from now on that $\rank(G) = \rank(K)$.

Theorem \ref{thm cts} implies a continuity property of characters of $L$-packets of discrete series reresentations. 

As before, let $T<K$ be a maximal torus, and
set $W_K := N_K(T)/T$. Let $W_G$ be the Weyl group of the root system of $(\kg^{\C}, \kt^{\C})$. Fix representatives $w \in W_G$ of all classes $[w] \in W_G/W_K$. For any discrete series representation with Harish-Chandra parameter $\lambda$, we denote its global character by $\Theta_{\lambda}$.
\begin{corollary}\label{cor char cts}
Let $\pi$ be a discrete series representation of $G$ with Harish-Chandra parameter  $\lambda\in i\mathfrak t^*$. Then
\[
\lim_{g \to e; g \in T^{\reg}}
\sum_{[w]\in W_G/W_K}\Theta_{w\lambda}(g) = d_{\pi}
\]
where 
 $d_{\pi}$ is the formal degree of $\pi$. 
\end{corollary}
This corollary will be proved after we prove Theorem \ref{thm cts}.
As a consequence, one can take the limit as $g \to e$ in Harish-Chandra's character formula to obtain an expression for $d_{\pi}$; see e.g.\ page 25 of \cite{Atiyah77}. See also Proposition 50 in \cite{GMW}.


\subsection{A $K$-theoretic character identity}



Let $G_c$ be a compact inner form of $G$, which exists because $\rank(G) = \rank(K)$. Inner forms are defined for example in Chapter 2 of \cite{ABV}, but the only properties we need are that 
$G_c$ is a real form of $G_{\C}$, and $T$ identifies with a Cartan subgroup of $G_c$. So pairs $(G,T)$ and $(G_c, T)$ have the same root system. The positive root system $R^+$ determines a $G$-invariant complex structure on $G_c/T$. For any integral $\nu \in i\kt^*$, consider the holomorphic line bundles
\[
\begin{split}
L^G_{\nu} &:= G\times_T \C_{\nu} \to G/T;\\
L^{G_c}_{\nu} &:= G_c\times_T \C_{\nu} \to G_c/T.
\end{split}
\]
Let $\bar\partial_{L^{G}_{\nu}}$ and $\bar\partial_{L^{G_c}_{\nu}}$ be the Dolbeault operators on $G/T$ and $G_c/T$, respectively, coupled to these line bundles. 

In \cite{HW3}, the authors prove a $K$-theoretic analogue of Shelstad's character identities \cite{Shelstad79}, and deduce Shelstad's character identity in the case of the discrete series.
\begin{theorem} \label{thm char id}
For all integral $\nu \in i\kt^*$ and all $g \in S$,
\[
\tau_g(\ind_{G_c}(\bar\partial_{L^{G_c}_{\nu}}+\bar\partial_{L^{G_c}_{\nu}}^*))=\sum_{[w]\in W_G/W_K} \tau_g(\ind_{G}(\bar\partial_{L^{G}_{w^{-1}\nu}}+\bar\partial_{L^{G}_{w^{-1}\nu}}^*)).
\]
\end{theorem}
\begin{proof}
This is (3.6) in \cite{HW3}. There, $\nu$ is regular but that property is not used in the proof of the above equality.
\end{proof}

\subsection{Dolbeault operators}

We will use some properties of the Dolbeault--Dirac operators in Theorem \ref{thm char id} to deduce Theorem \ref{thm cts}. 

First of all,  every element of $K_0(C^*_rG)$ is the index of a Dolbeault--Dirac operator on $G/T$.
Indeed, 
let $V \in \hat K_{\Spin}$, and let $\lambda \in i\kt^*$ be its highest weight with respect to the positive compact roots chosen earlier. Then $\lambda - \rho_n$ is a weight of $\Delta_{\kp} \otimes V$, so it is integral for $T$.
Consider the holomorphic, $G$-equivariant line bundle
\[
L^G_{\lambda- \rho_n}:= G\times_T \C_{\lambda - \rho_n} \to G/T.
\]
Let $\bar \partial_{L^G_{\lambda-\rho_n}}$ be the Dolbeault operator on $G/T$ coupled to $L^G_{\lambda - \rho_n}$.
\begin{proposition} \label{prop GT}
We have
\[
\DInd_K^G[V_{\lambda}] = (-1)^{\dim(G/K)} \ind_G(\bar \partial_{L^G_{\lambda-\rho_n}} + \bar \partial_{L^G_{\lambda-\rho_n}}^*).
\]
\end{proposition}
\begin{proof}
This is proved in Section 5 of \cite{HW} in the case where $\lambda+ \rho_c$ is regular for $G$, but that assumption is not necessary for the arguments.
\end{proof}

\begin{lemma}\label{lem Lw}
We have for all $w \in W_G$ and all $g \in S$,
\[
\tau_{wgw^{-1}}(\ind_G(\bar\partial_{L^G_{\lambda-\rho}}+\bar\partial_{L^G_{\lambda-\rho}}^*))
=
\tau_g(\ind_{G}(\bar\partial_{L^{G}_{w^{-1}(\lambda-\rho)}}+\bar\partial_{L^{G}_{w^{-1}(\lambda-\rho)}}^*)).
\]
\end{lemma}
\begin{proof}
In the case of Dolbeault operators twisted by holmorphic vector bundles, and finite fixed point sets, the fixed point formula in Theorem \ref{thm fixed pt} simplifies considerably, see Corollary 6.3 in \cite{HW}. For any $h \in T$ with dense powers, and any integral $\nu \in i\kt^*$, this yields
\beq{eq fixed Dolb}
\tau_{h}(\ind_G(\bar\partial_{L^G_{\nu}}+\bar\partial_{L^G_{\nu}}^*)) = 
\sum_{xT \in (G/T)^{h}} \frac{\tr(g|_{(L^G_{\nu})_{xT}})}{\det_{\C}(1-g^{-1}|_{T_{xT}G/T})}.
\eeq
Now, for $w \in W_G$, we have $(G/T)^{wgw^{-1}} = (G/T)^T = N_K(T)/T$. And for $x \in N_K(T)$,
\[
\begin{split}
(L^G_{\nu})_{xT} &= \C_{\Ad^*(x)\nu};\\
T_{xT}G/T &= \bigoplus_{\alpha \in R^+} \C_{\Ad^*(x)\alpha}.
\end{split}
\]
as complex representations of $T$, where we use the complex structure on $G/T$ defined by $R^+$. So
\begin{align}
\tau_{wgw^{-1}}(\ind_G(\bar\partial_{L^G_{\lambda-\rho}}+\bar\partial_{L^G_{\lambda-\rho}}^*))&=
\sum_{xT \in N_K(T)/T} \frac{\tr(wgw^{-1}|_{\C_{\Ad^*(x)(\lambda-\rho)} })}{\det_{\C}(1-wg^{-1}w^{-1}|_{\bigoplus_{\alpha \in R^+} \C_{\Ad^*(x)\alpha}})} \nonumber \\
&= \sum_{xT \in N_K(T)/T} \frac{\tr(g|_{\C_{\Ad^*(w^{-1}x)(\lambda-\rho)} })}{\det_{\C}(1-g^{-1}|_{\bigoplus_{\alpha \in R^+} \C_{\Ad^*(w^{-1}x)\alpha}})} \nonumber \\
&= \sum_{yT \in w^{-1}N_K(T)w/T} \frac{\tr(g|_{\C_{\Ad^*(yw^{-1})(\lambda-\rho)} })}{\det_{\C}(1-g^{-1}|_{\bigoplus_{\alpha \in R^+} \C_{\Ad^*(yw^{-1})\alpha}})}.\label{eq tau wgw}
\end{align}
(In the last step, we substituted $y = w^{-1}xw$.)

Finally, $w^{-1}N_K(T)w = N_K(T)$, and 
\[
\bigoplus_{\alpha \in R^+} \C_{\Ad^*(yw^{-1})\alpha} = T_{yT}G/T
\]
as complex representations of $T$, with respect to the complex structure defined by the positive root system $w^{-1}R^+$ with respect to which $w^{-1}(\lambda - \rho)$ is dominant.
So by \eqref{eq fixed Dolb}, the expression \eqref{eq tau wgw} equals
\[
\tau_g(\ind_{G}(\bar\partial_{L^{G}_{w^{-1}(\lambda-\rho)}}+\bar\partial_{L^{G}_{w^{-1}(\lambda-\rho)}}^*)).
\]
\end{proof}

 \begin{lemma}
\label{lem:2}
We have for all integral $\nu \in i\kt^*$,
\[
\tau_e(\ind_{G_c}(\bar\partial_{L^{G_c}_{\nu}}+\bar\partial_{L^{G_c}_{\nu}}^*))=
\tau_e(\ind_{G}(\bar\partial_{L^{G}_{\nu}}+\bar\partial_{L^{G}_{\nu}}^*)).
\]
\end{lemma} 
\begin{proof}
By Connes--Moscovici's $L^2$-index formula, Theorem 5.2  in~\cite{Connes82}, we have
\[
\begin{split}
\tau_e(\ind_{G}(\bar\partial_{L^{G}_{\nu}}+\bar\partial_{L^{G}_{\nu}}^*)) &= \varepsilon(\ch(\Bigwedge_{\C}\kg/\kt \otimes \C_{\nu}) \hat A(\kg, T))[\kg/\kt];\\
\tau_e(\ind_{G_c}(\bar\partial_{L^{G_c}_{\nu}}+\bar\partial_{L^{G_c}_{\nu}}^*)) &=
\varepsilon(\ch(\Bigwedge_{\C}\kg_c/\kt \otimes \C_{\nu}) \hat A(\kg_c, T))[\kg_c/\kt],\\
\end{split}
\]
for the same sign $\varepsilon \in \{\pm1\}$.
Here  $\ch\colon R(T)\rightarrow H^*(\mathfrak g, T, \R)$ is the relative Chern character, and the characteristic classes $\hat A$ in $H^*(\mathfrak g, T, \R)$ are defined in Section 4 of \cite{Connes82}. The right hand side of the first line only depends on the representations $\Bigwedge_{\C}\kg/\kt \otimes \C_{\nu}$ and $\kg/\kt$ of $T$, and similarly for the right hand side of the second line. Since $\kg/\kt$ and $\kg_c/\kt$
 both equal the sum of the positive root spaces as complex representations of $T$, we find that the two expressions are equal.
\end{proof}

 \subsection{Proofs of Theorem \ref{thm cts} and Corollary \ref{cor char cts}}
 
To finish the proof of Theorem~\ref{thm cts}, we need a final lemma. 
\begin{lemma}[Arthur]\label{lem Arthur}
We have for all $g \in T^{\reg}$,
\[
\tau_g^s = \sum_{[w] \in W_G/W_K} \tau_{wgw^{-1}}.
\]
\end{lemma}
\begin{proof}
In Section 27 (p.194) of \cite{Arthur}, it is pointed out that two elements $g,g' \in T^{\reg}$ are conjugate if and only if $g=w_Kg'w_K^{-1}$ for some $w_K\in W_K$, and stably conjugate if and only if $g=w_Gg'w_G^{-1}$ for some $w_G\in W_G$.
\end{proof}


\begin{proof}[Proof of Theorem~\ref{thm cts}]
By surjectivity of Dirac induction and Proposition \ref{prop GT}, 
every $x\in K_0(C^*_rG)$ is represented by the equivariant index 
\[
x=\ind_G(\bar\partial_{L^G_{\nu}}+\bar\partial_{L^G_{\nu}}^*)
\]
for an integral element $\nu \in i\kt^*$.

Let $g \in S$. 
By Theorem \ref{thm char id} and Lemmas  \ref{lem Lw} and \ref{lem Arthur}, we have
\[
\tau_g^s(x)=\tau_g^s(\ind_G(\bar\partial_{L^G_{\nu}}+\bar\partial_{L^G_{\nu}}^*))
=\tau_g(\ind_{G_c}(\bar\partial_{L^{G_c}_{\nu}}+\bar\partial_{L^{G_c}_{\nu}}^*)).
\]
Since $G_c$ is compact, this expression is continuous in $g$. And
by Lemma~\ref{lem:2},
\[
\tau_e(\ind_{G_c}(\bar\partial_{L^{G_c}_{\nu}}+\bar\partial_{L^{G_c}_{\nu}}^*))
=\tau_e(\ind_{G}(\bar\partial_{L^{G}_{\nu}}+\bar\partial_{L^{G}_{\nu}}^*))=\tau_e(x).
\]
\end{proof}
 
\begin{proof}[Proof of Corollary \ref{cor char cts}]
For $w \in W_G$, let $[\pi_{w\lambda}] \in K_0(C^*_rG)$ be the class defined by the discrete series representation with Harish-Chandra parameter $w\lambda$.
By Propositions 5.1 and 5.2 in \cite{HW2}, we have for 
 all $g \in T^{\reg}$,
\[
\begin{split}
\sum_{[w]\in W_G/W_K}\Theta_{w\lambda}(g) &= 
\sum_{[w]\in W_G/W_K}\tau_{g}([\pi_{w\lambda}]) \\
&= (-1)^{\dim(G/K)/2}
\sum_{[w]\in W_G/W_K} \tau_g(\ind_{G}(\bar\partial_{L^{G}_{w(\lambda-\rho)}}+\bar\partial_{L^{G}_{w(\lambda-\rho)}}^*)).
\end{split}
\]
Lemmas \ref{lem Lw} and \ref{lem Arthur} imply that the right hand side equals
\[
(-1)^{\dim(G/K)/2}
\tau_g^s(\ind_{G}(\bar\partial_{L^{G}_{\lambda-\rho}}+\bar\partial_{L^{G}_{\lambda-\rho}}^*)).
\]
As $g \to e$ through the set  $S$
in Theorem \ref{thm cts}, that result implies that the limit of the above expression is
\[
(-1)^{\dim(G/K)/2}
\tau_e (\ind_{G}(\bar\partial_{L^{G}_{\lambda-\rho}}+\bar\partial_{L^{G}_{\lambda-\rho}}^*)) = \tau_e([\pi_{\lambda}]) = d_{\pi}.
\]
The claim now follows from continuity of characters on the regular set.
\end{proof}

\bibliographystyle{plain}

\bibliography{mybib}

\end{document}